\documentclass[12pt]{article}
\usepackage{amsmath,amsthm,amsfonts,amssymb,verbatim}
\usepackage[usenames]{color} 
\newtheorem{thm}{Theorem}[section]
\newtheorem{lemma}[thm]{Lemma}
\newtheorem{defn}[thm]{Definition}
\newtheorem{cor}[thm]{Corollary}

\newtheorem{prop}[thm]{Proposition}

\theoremstyle{definition}

\newtheorem{rem}[thm]{Remark}

\newtheorem{remark}[thm]{Remark}
\theoremstyle{remark}
\DeclareMathOperator{\card}{Card}
\DeclareMathOperator{\dlen}{DLen}
\DeclareMathOperator{\Fix}{Fix}

\DeclareMathOperator{\Homeo}{Homeo}

\newcommand{\B}{{\cal B}}
\newcommand{\R}{\mathbb R}
\newcommand{\E}{{\cal E}}

\newcommand{\G}{{\cal G}}
\newcommand{\K}{{\cal K}}

\newcommand{\T}{\mathbb T}

\newcommand{\Z}{\mathbb Z}

\def\A{{\mathbb A}}
\def\C{{\cal C}}

\def\Q{{\mathbb Q}}

\def\U{{\cal U}}

\def\ti{\tilde}

\def\sl3z{SL(3, \mathbb Z)}

\def\G{{\cal G}}
\def\H{{\cal H}}

\def\M{{\cal M}}

\def\A{{\mathbb A}}
\def\Q{{\mathbb Q}}
\def\cS{{\cal S}}

\newcommand\jmf[1]{\textcolor{black}{#1}}  
\newcommand\jmfc[1]{\textcolor{black}{#1}}

\title{Zero entropy subgroups of mapping class groups}

\author{John Franks and Kamlesh Parwani \thanks{This work was partially supported by a grant from the Simons Foundation (\#280151 to Kamlesh Parwani).}}

\begin{document}
\maketitle
\begin{abstract}
Given a group action on a surface with a finite invariant set we
investigate how the algebraic properties of the induced group of
permutations of that set affects the dynamical
properties of the group.  Our main result shows that in many
circumstances if the induced permutation group is not solvable then
among the homeomorphisms in the group there must be one with a
pseudo-Anosov component. We formulate this in terms of the mapping
class group relative to the finite set and show the stronger result
that in many circumstances (e.g. if the surface has boundary) if this
mapping class group has no elements with pseudo-Anosov components
then it is itself solvable.
\end{abstract}

\section{Introduction} 

Let $M$ be a compact surface with boundary.  We are interested in the
question of how a group action on $M$ permutes a finite invariant set $X
\subset int(M)$.  More precisely, how the algebraic properties of the
induced group of permutations of a finite invariant set affects the
dynamical properties of the group.  Our main result shows that in many
circumstances if the induced permutation group is not solvable then
among the homeomorphisms in the group there must be one with a
pseudo-Anosov component. We formulate this in terms of the mapping
class group relative to the finite set and show the stronger result
that in many circumstances (e.g. if $\partial M \ne \emptyset$) this
mapping class group is itself solvable if it has no elements with
pseudo-Anosov components.

\begin{defn}  By $\B(M, X)$ we denote the mapping
class group of $M \setminus X$, where $X$ is a finite,
non-empty set.  In the notation of \cite{FM}, if
$S = M \setminus X$ then $\B(M, X) = Mod(S, \partial S)$. 
\end{defn} 

So an element of $\B(M, X)$ is an isotopy class of homeomorphisms of
$M$ which fix $\partial M$ pointwise and permute the elements of
$X$. The isotopies as well as the homeomorphisms are required to fix
$\partial M$ pointwise.

We will say that an element $g$ of $\B(M, X)$ has {\em zero entropy} if it
has a representative with zero entropy. An equivalent and very useful way to 
say this is that $g$ has a Thurston canonical form
with no irreducible components of pseudo-Anosov type
(see Theorem \ref{Thurston} below).
We will say that \jmfc{a subgroup} $G$ of  $\B(M, X)$ has {\em zero entropy} if 
each element has entropy zero. For any subgroup $G$ of  $\B(M, X)$ there is
a natural homomorphism $\pi: G \to \cS_G(X)$ where $\cS_G(X)$ is the group 
of permutations of $X$ induced by $G$. We will denote the derived length of a solvable
group $G$ by $\dlen(G)$.  

We first consider the case of surfaces with genus zero.

\begin{thm}\label{solvable-0}
  Suppose $M$ is an oriented compact connected surface of genus $0$
  and let $G$ be a finitely generated infinite subgroup of $\B(M, X)$
  with zero entropy, where $X$ is a non-empty finite
  set. \jmfc{If} $\partial M \ne \emptyset,$ or there is a point $x \in X$
  fixed by all elements of $G$, then $G$ is solvable. If $M = S^2$
  then there is a finite index normal subgroup $G_0$ of $G$ which is
  solvable and such that $G/G_0$ acts effectively on $S^2.$ Moreover
  when $G$ is solvable the kernel $K$ of the homomorphism
  $\pi: G \to \cS_{G}(X)$ is free abelian.
\end{thm}

\begin{cor}\label{cor-0}
Suppose $G$ satisfies the hypothesis of Theorem~\ref{solvable-0}.
If $G$ is solvable then its derived length $\dlen(G)$ satisfies 
\[
\dlen(G) -1 \le \dlen(\cS_{G}(X)) \le \dlen(G).
\]
\end{cor}

\begin{rem}
Note that as a consequence if  $G$ satisfies the hypothesis of Theorem~\ref{solvable-0},
it is virtually abelian. This follows by applying Theorem~\ref{solvable-0}
to the solvable group $G_0$ and noting
$G_0$ has finite index in $G$ and $G_0/K$ is finite.
This is also a consequence of Theorem B of \cite{BLM} and results presented in \cite{Ivanov}.
\end{rem}

An interesting question asks which permutations of a finite
set force a braid inducing that permutation on its ends to have
a pseudo-Anosov component. 
An interesting corollary of Theorem~\ref{solvable-0} gives a necessary
algebraic condition, not for a single permutation, but for a subgroup of
the permutation group $S_n$ induced by a subgroup of the braid
group on $n$-strands.  

\begin{cor}Suppose $G$ is a finitely generated subgroup of the 
braid group on $n$-strands and $\pi: G \to S_n$ is the natural
projection onto the permutation group of the ends of the strands.
If $G$ has entropy zero then it is solvable and the kernel of  $\pi$ is free abelian
so $G$ is virtually abelian and  has derived length 
satisfying $\dlen(G) -1 \le \dlen(\pi(G)) \le \dlen(G).$
\end{cor}

When the genus of the surface $M$ is positive we have the following result.
As in Theorem~\ref{solvable-0} we assume $X$ is a non-empty finite
set.

\begin{thm}\label{solvable-1}
Suppose $M$ is an oriented compact connected surface of genus $g >
1$ and let $G$ be a finitely generated infinite subgroup of $\B(M, X)$ which
contains no element with positive entropy. 
\begin{enumerate}
\item If $\partial M \ne \emptyset$, or there exists $x \in X$ fixed by each
element of $G$, then $G$ is solvable.
\item Otherwise $G$ contains a solvable 
subgroup $G_0$ of index at most $84(g-1)(2g-2)!$ in $G$.
\end{enumerate}
\end{thm}

There is also a result for the case of the torus.
We continue with the assumption that $X$ is a non-empty finite
set.

\begin{thm}\label{solvable-T}
Suppose $M$ is an oriented compact connected surface of genus $g = 1$ 
and let $G$ be a finitely generated subgroup of $\B(M, X)$ which
contains no element with positive entropy. Then $G$ is solvable.
\end{thm}

\section{Graph Maps} 

We will have need of an elementary result about groups acting on
a finite tree.

\begin{prop} \label{tree action}
 Let $G$ be a group acting isometrically on
a finite tree.  Then $\Fix(G) \ne \emptyset.$
\end{prop}
\begin{proof}
Let $T$ be a finite tree. If $e$ is an edge of $T$, which 
has been assigned an orientation, we will denote by $T^+(e)$
and $T^-(e)$ the two subtrees whose union is the complement
of the interior of $e$. The subtree $T^+(e)$ is chosen to
contain the positive endpoint of $e$ and $T^-(e)$ the
negative.

We wish to orient the edges of $T$ so that for every edge
$e$ the subtree $T^+(e)$ contains more than half of the
vertices of $T$.  This is possible unless there is an
edge $e_0$ such that $T^+(e_0)$ and $T^-(e_0)$ each contain
exactly half the vertices of $T$.  Such an edge is unique,
if it exists, because every other edge lies in (and separates) either 
$T^+(e_0)$ or $T^-(e_0)$.  

Since the edge $e_0$ is unique, if it exists, it must be preserved by
each element of $G$ and hence the midpoint of $e_0$ is an element of
$\Fix(G).$

We are left with the case that every edge of $T$ can be
oriented in such a way that 
the subtree $T^+(e)$ contains more than half of the vertices
of $T$.  Note that this implies that at most one edge can
exit from any vertex.  If one starts at any point on the tree
and follows a path consistent with the orientation this path
will necessarily end in a sink vertex $v_0$ (i.e. one with
no exiting edges) since no cycles are
possible.

We claim that this sink is unique.  This is because its
``basin of attraction,'' i.e. the union of all edges on paths
which lead to $v_0$, is a subtree of $T$.  If there were more
than one such subtree two would intersect and their intersection
would contain a vertex with more than one exiting edge. 

Since $v_0$ is unique it must be a fixed point for every 
element of $G$.
\end{proof}

\section{Dehn Twists}

We denote $\R/\Z$ by $\T^1.$

\begin{defn} A {\em generalized Dehn twist} between two copies
of  $\A = \T^1 \times I$, say $\A_0$ and $\A_1$
is a homeomorphism $\alpha : \A_0 \to \A_1$ 
which for some $a,b \in \Q$ is given by 
$\alpha(x, t) = ( x + at +b \ (mod\ 1),\ t)$.  A generalized
Dehn twist is called {\em essential} if $a \ne 0.$ 
\end{defn}

We will frequently be interested in the case that $\A_0 = \A_1$
Observe that in this case the generalized Dehn twists on $\A = \A_0 =
\A_1$ form an abelian group. Note that since $a,b \in \Q$, \ 
 $\alpha: A \to \A$ has finite order on the boundary of $S^1 \times I.$

\begin{lemma}\label{Dehn twist}
Suppose $f: \A_0 \to \A_1$  is a homeomorphism which is 
a \jmfc{rational rotation} on each component of $\partial \A_0$.  Then there is a 
unique generalized Dehn twist $\alpha(f)$ on $\A_0$ which is isotopic to $f$
relative to $\partial \A_0.$  If $g$ is another such homeomorphism then
$\alpha(f \circ g) = \alpha(f) \circ \alpha(g) = 
\alpha(g) \circ \alpha(f) = \alpha(g \circ f).$
\end{lemma}

\begin{proof}
Let $\ti \A = \R \times I$ be the universal cover of $\A.$
Choose $b_0 \in \R$ which projects to the rotation $f|_{\T \times \{0\}}$ on
$\T$.  Let $\ti F: \ti \A \to \ti \A$ be the lift for which
$\ti f|_{\R \times \{0\}}$ is translation by $b_0.$ Then 
$\ti f|_{\R \times \{1\}}$ is translation by some $b_1 \in \R.$ 
We define $\ti \alpha: \ti \A \to \ti \A$  by
$\ti \alpha( x , t) = x + at + b_0$ where $a = b_1 - b_0$. 
The homeomorphism $\ti \alpha$ is equivariant with respect to the 
$\Z$ action on $\ti \A$ so $\ti \alpha$ is the lift of a generalized
Dehn twist $\alpha(f).$  If $g$ is another homeomorphism with the
properties of $f$ then from the construction it is clear that
$\alpha(f \circ g) = \alpha(f) \circ \alpha(g).$
Sinc the group of generalized Dehn twists is abelian
it follows that $\alpha(f \circ g) =  \alpha(g \circ f).$
\end{proof}

\section{Thurston's Theorem}

We cite a special case of Thurston's theorem on classification of
surface homeomorphisms up to isotopy. We limit ourselves to those 
homeomorphisms with zero entropy or equivalently have no pseudo-Anosov
components as this is appropriate for our needs.  \jmf{Throughout this
section $\Homeo_0(M)$ denotes the orientation preserving homeomorphisms
of the compact surface $M$.}

\begin{thm}[Thurston]\label{Thurston}
Suppose $M$ is a compact oriented surface and $X \subset M$ is finite.
Let  $[g]$ be an element of $\B(M, X)$ which has a representative with
entropy zero. There exists a \jmf{homeomorphism} $f$ representing $[g]$
and a finite set of pairwise disjoint closed annuli
$\U = \{ A_j\}$  in $M \setminus X$ with the following properties:
\begin{enumerate}
\item For each component $S$ of $\partial M$ there is an annulus $A \in \U$ 
such that $S$ is one of the boundary components of $A$.
\item The diffeomorphism $f$ permutes the elements of $\U = \{A_i\}$
\item The restriction of $f$ to each $A_i$ is a
generalized Dehn twist from $A_i$ to its image.  This twist is essential
if $A_i$ does not contain a component of $\partial M.$
 
\item If $Y$ is a component of the complement of $U = \bigcup_i A_i$ which is
invariant under $f^k$ then $f^k|_{Y}$ has finite order.
\end{enumerate}
\end{thm}

\begin{remark}
The isotopy class $[c_i]$ of an essential simple closed curve $c_i$ in $A_i$
which is not isotopic to a component of $\partial M$ 
is called an {\em essential reducing curve}. The curves $c_i$ and 
the annuli $A_i$ are unique up to isotopy (see \cite{BLM}).
Each of the neighborhoods $A_i$ may be chosen to lie in an arbitrarily
small neighborhood of a component of $\partial M$ or of an essential
reducing curve $c_i.$
\end{remark}

We will say that two embedded annuli $A_1$ and $A_2$ {\em intersect essentially}
in $M \setminus X$ if any embedded simple closed curves $c_1$ and $c_2$,
which are isotopic to core curves of $A_1$ and $A_2$ respectively, must
have a point of intersection. We will say they {\em intersect minimally} if
there intersection has the fewest number of components among all choices
of isotopy class representatives for $A_1$ and $A_2$.

\begin{lemma}\label{+entropy}
\jmf{Let $M$ be a compact oriented surface and let $X\subset M$ 
be a finite subset.  Suppose $[g_1], [g_2] 
\in \B(M,X)$ each satisfy the hypothesis of  Theorem~\ref{Thurston}, i.e.,
they contain zero entropy representatives. For $i = 1,2$
let $f_i$ denote the homeomorphism guaranteed by this Theorem
and \jmfc{let} $\U_i$ denote the associated family of annuli. Then if $U_i(g_1) \in
\U(g_1)$ and $U_j(g_2) \in \U(g_2)$ intersect essentially, the group
$\G$ generated by $[f_1] = [g_1]$ and  $[f_2] = [g_2]$ 
in  $\B(M,X)$ has elements each of whose representatives has positive entropy.}
\end{lemma}

\begin{proof}
\jmf{Since $[f_i] = [g_i]$ we may without loss of generality choose $g_i= f_i$.}
By taking powers if necessary we may assume $g_i$ preserves
the components of $\U(g_i)$.
Theorem 4.6 of \cite{McP}  (see also \cite{FM2} for the result as stated below)
implies that any subgroup $G$ of $\B(M,X)$ satisfies one of the following:

\begin{enumerate}
\item $G$ contains a pair of independent pseudo-Anosov elements 
(such subgroups are called sufficiently large in \cite{McP}]),
\item $G$ fixes the the laminations $\{F^-(g); F^+(g)\}$ for a certain pseudo-Anosov element g,
\item $G$ is and infinite subgroup fixing (up to isotopy) a finite
  disjoint system of non-peripheral essential simple closed curves on M (such
  subgroups are called reducible),
\item $G$ is finite.
\end{enumerate}

If $G$ is the subgroup generated by $[g_1]$ and $[g_2]$ then it
contains elements of infinite order (e.g. $[g_i]$) so (4) is
impossible. In cases (1) and (2) the group $G$ contains elements which
are isotopic to pseudo-Anosov maps and \jmf{hence  any representative of
such an element has positive entropy.}

We are left with case (3), the reducible case. In this case we claim
that no element of the finite set $C$ of non-peripheral essential
simple closed curves invariant under $G$ can have an essential
intersection with an element of $\U(g_i)$.

To show this claim we choose a hyperbolic metric on $M$ and geodesic
representatives $\{c_j\}$ of the elements of $C$ and also choose
geodesic representatives $\{u_k\}$ of $\U(g_i)$.  We assume that
$c_0$ and $u_0 \in \U(g_1)$ have an essential intersection and show
this leads to a contradiction, the other cases being similar.  Let $h$
be isotopic to a power of $g_1$ which twists more than once around
$u_0$ and is chosen so that $h = id$ outside of the annuli
$\U(g_1)$. Let $\ti u_0$ be a lift of $u_0$ to the universal cover
$\ti M$ and let $\ti U_0(g_1)$ be the lift of $U_0(g_1)$ containing
$\ti u_0$.  Since $u_0$ and $c_0$ intersect essentially there is a
lift $\ti c_0$ which intersects $\ti u_0$.  There are lifts of $h$
which fix the enpoints $p^+, p^-$ of $\ti u_0$ in the compactification
of $\ti M$ obtained by adding the circle at infinity $S_\infty$.
These lifts all differ by compositiion with a power of $T$, the covering
translation corresponding to $c_0$. Because $h$ twists more than one
turn there is a lift $\ti h$ of $h$ which preserves $\ti U_0(g_1)$ and
moves points in the boundary components of of $\ti U_0(g_1)$ in
opposite directions, i.e. for $z$ in one component of $\partial \ti U_0(g_1)$ we
have $\ti h(z) = T^d(z)$ for some $d >0$ while for  $z'$ in the other
component $\ti h(z') = T^{d'}(z)$ for some $d'<0$. 

It follows that $\ti h$ fixes $p^+$ and $p^-$ and
moves the points in the components of
$S_\infty \setminus \{p^+, p^-\}$ in opposite directions. More
precisely, we will show, that for some $x_0$ in one component of
$S_\infty \setminus \{p^+, p^-\}$ we have $\ti h^n(x_0)  = T^{nd}(x_0)$ and
hence for 
any $x$ in the same componenent of  $S_\infty \setminus \{p^+, p^-\}$
will have
\[
\omega(x,\ti h) = \lim_{n \to \infty} \ti h^{n}(x_0) = \lim_{n \to \infty} T^{nd}(x_0) = p^+,
\]
and similarly for $x$ in the other component we have
$ \omega(x,\ti h) = p^-$. To see this consider a curve $\ti \gamma(t)$
in $\ti M$ starting at a point
$\gamma(0) = z \in \partial \ti U_0(g_1)$ and ending at a point of
$S_\infty \setminus \{p^+, p^-\}$ but with
$\gamma(t) \notin \ti U_i(g_1)$ for all $i$ and all $t>0$.  Then the
projection of $\gamma(t)$ to $M$ lies in the fixed point set of $h$,
so the lift $\ti h$ must take $\ti \gamma(t)$ to $T_0(\ti \gamma(t)$
for some covering translation $T_0$ independent of $t$.  Since
$\ti h(z) = T^d(z)$ it must be that $T_0 = T^d$ and
$\ti h^n(x_0) = T^{nd}(x_0)$ where $x_0$ is the end of $\gamma(t)$ in
$S_\infty$.

To produce the contradiction giving us the claim we first note that
since the geodesics $c_0$ and $u_0$ intersect essentially
there is a lift $\ti c_0$  of $c_0$ which intersects  $\ti u_0$.
The ends of $\ti c_0$ are then in
different components of $S_\infty \setminus \{p^+, p^-\}$.  We know
that $h(c_0)$ is isotopic to $c_0$. Hence the lifted geodesic $\ti c'$
whose ends are the ends of $\ti h(\ti c_0)$ is another lift of $c_0$
and there must be a covering translation $T_1$ such that
$T_1(\ti c_0) = \ti c'$. But since $\ti h$ moves ends of $\ti c_0$ in
opposite directions the ends of $\ti c'$ are separated in
$\ti M \cup S_\infty$ by $\ti c_0$.  This would require that $\ti c'$
and $\ti c_0$, two lifts of $c_0$, intersect.  This contradiction
proves the claim that no element of the finite set $C$ of non-peripheral essential
simple closed curves invariant under $G$ can have an essential
intersection with an element of $\U(g_i)$.

We can now complete the proof of the lemma.  Since all elements
of $C$ have representatives disjoint from $\U(g_1)$ and $\U(g_2)$, we may reduce
along the curves of $C$ obtaining an induced subgroup of the mapping
class of their complement.  One component $P$ of this complement
contains elements of both $\U(g_1)$ and $\U(g_2)$ which intersect
essentially. This component must be invariant under $[g'_1]$ and $[g'_2]$
where $[g'_i]$ is the element of the mapping class of $P \setminus X$
induced by $[g_i]$.
Hence we may consider the maps $g'_1$ and $g'_2$ and 
and the subgroup $G_P$ of the mapping class group of $P \setminus X$ which they
generate. We again must have one of the four alternatives listed
above.  Case (4) is not possible since $g_i$ has infinite order.  Case
(3) is not possible since we have reduced as much as possible.  Cases
(1) and (2) imply there is an element $[g] \in G$ such that mapping class
of $P \setminus X$ which it induces 
is pseudo-Anosov. It follows that $[g]$ has a pseudo-Anosov
componenent and in particular any $g$ representing it has positive
entropy.
\end{proof}

\begin{lemma}\label{Kerckhoff} 
Suppose $M$ is a compact oriented surface, $X \subset int(M)$ is a
non-empty finite subset, and $M \setminus X$ has negative Euler
characteristic and is provided with a hyperbolic metric.  Suppose also
that $G$ is a finitely generated subgroup $\B(M, X)$ every element of
which has finite order.  Then there is a finite group $\G$ of
isometries of $M \setminus X$ and an isomorphism $\psi: G \to \G$ such
that for each $[h] \in G$ the homeomorphism $\psi([h])$ is in the
isotopy class $[h].$
\end{lemma}

\begin{proof}
This is almost a result of Kerckhoff (see \cite{K} or Theorem 7.2 of
\cite{FM}), but we need to know that $G$ is a finite group.  For this
we follow a proof given in Chapter 7 of \cite{FM}.  The action of $G$
on $H_1(M \setminus X, \R)$ is linear and hence its image is finite by
the classical Burnside result.  Hence $G$ has a finite index subgroup
$G_0$ which acts as the identity on $H_1$. If $[h] \in G_0$ then it
has finite order and can be realized by a finite order isometry of $M
\setminus X$ (this follows, e.g., from Kerckhoff's theorem.  See
\cite{K} or Theorem 7.2 of \cite{FM}). We would like to apply the
Lefschetz fixed point theorem which requires a compact space, so we
let $M_0$ be the complement in $M$ of small open $h$-invariant
neighborhoods of points of $X$.  The action of $h$ is trivial on
$H_1(M_0).$ Since $h$ is an isometry of $M_0$, if $h \ne id,$ then its
fixed points are isolated and each fixed point has Lefschetz index
$1$.  It follows that $\card(\Fix(f)) = L(f)$ where $L(f)$ is the
Lefschetz number of $f$.  By the Lefschetz fixed point theorem $L(f) =
tr(f_{*0}) - tr(f_{*1}) = 1 - \dim(H_1(M_0))$ (because $tr(f_{*2}) =
0$).  Since $M_0$ has negative Euler characteristic $\dim(H_1(M_0)) \ge
2$ which implies $L(f) < 0$ so $L(f) \ne \card(\Fix(f)).$ This
contradiction implies $h = id.$ We conclude that $G_0$ is trivial and
$G$ is finite.

Now by Kerckhoff's  theorem
(see \cite{K} or Theorem 7.2 of \cite{FM}) $G$ can be realized
by a group $\G$ of isometries of a hyperbolic metric.
\end{proof}

\begin{lemma}\label{C-properties}
Suppose $M$ is a compact oriented surface and $X \subset M$ is finite.
Let  $G$ be a finitely generated
subgroup of $\B(M, X)$ each element of which has a representative of zero
entropy. There exists a set $\C$ of simple closed curves in $M \setminus X$
such that
\begin{enumerate}
\item Each $c \in \C$ is a representative of an isotopy class of an essential
reducing curve for some element of $G$ and each such isotopy class is represented
by some element of $\C.$
\item Elements of $\C$ are pairwise disjoint.
\item The set $\C$ is invariant up to isotopy in $M \setminus X$ under $G$.
\item If $\C \ne \emptyset$ each component of the complement of $\C$ in $M \setminus X$
has negative Euler characteristic.
\item  $\C$ is finite.
\end{enumerate}
\end{lemma}

\begin{proof}
Let $\bar \C = \{ [c_j]\}$ be the set of 
isotopy classes of all essential reducing curves for all elements of $G$
acting on $M \setminus X$.  We may choose the representative  $c_i$  of $[c_i]$ so
that $c_i \cap c_j = \emptyset$ if  $i \ne j$ since otherwise
there would be an element with positive entropy by Lemma~\ref{+entropy}. 
Let $\C = \{c_j\}$ then (1) and (2) are satisfied.

The set $\C$ is invariant up to isotopy under $G$ since if $c_i$ is a reducing curve
for $f$ then $g(c_i)$ is a reducing curve for $g f g^{-1}.$  This proves
(3). 

If $\C \ne \emptyset$ and $Y$ is a component of the complement of $\C$ in 
$M \setminus X$ with non-negative Euler characteristic then $Y$ is homeomorphic
to the annulus or the disk.  If it were the annulus the two boundary components
of $Y$ would be isotopic in $M \setminus X$ and hence represent the same
element of $\C.$ If it were a disk its boundary would be a null-homotopic 
reducing curve for some element of $G$ which does not occur by definition.
Hence we have shown (4).

The set $\C$ must be finite since if the complement of $\C$ has more than $n$ 
components each with negative Euler characteristic,  then the Euler
characteristic of $M \setminus X$ is at most $-n.$ But this
Euler characteristic is $\chi(M) - \card(X)$ and hence one calculates that
$n \le \card(X) -\chi(M) = \card(X) + 2g -2$ where $g$ is the genus of $M$.
This proves (5).
\end{proof}

The following result is implicit in Theorem 4.6 of \cite{McP}. Since
it largely follows from results of Thurston and Kerckhoff, we give the proof
here.

\begin{lemma}\label{canonical}
Suppose $M$ is a compact oriented surface and $X \subset M$ is finite.
Let  $G$ be a finitely generated
subgroup of $\B(M, X)$ each element of which has a representative of zero
entropy. There exists a finite set $\U$ of pairwise disjoint closed annuli
in $M \setminus X$ and an isomorphism $\phi: G \to \G$ where $\G$
is a subgroup of $\Homeo_0(M)$ 
such that
\begin{enumerate}

\item The elements of $\U$ are permuted by $\G$ and the restriction of
each $g \in \G$ to an $A_i \in \U$ is a generalized Dehn twist.

 \item For each $A_i \in \U$ there is some $g \in \G$ such that $g$ leaves $A_i$ invariant and $g | _{A_i}$ is an essential generalized Dehn twist.  Conversely for any 
non-trivial element of $G$ each of the annuli $A_j(g)$ guaranteed by 
Theorem~\ref{Thurston} is isotopic to one of the $A_i \in \U.$

 \item The restriction of $\G$ to $U = \bigcup_i A_i$ is a finitely generated
subgroup of $\Homeo_0(U).$  The stabilizer in $\G$ of any $A_i \in \U$ is
infinite cyclic. 
 \item The restriction of $\G$ to $M \setminus U$ is 
 a finite group. 
\end{enumerate}
\end{lemma}

\begin{proof}
Let $\C$ be the set of simple closed curves guaranteed by Lemma~\ref{C-properties}

We consider the set $\U = \{A_i\}$ of pairwise disjoint annular neighborhoods of representatives of each $c_i$ and of each component of $\partial M$. We choose these so that
$A_i \cap X =\emptyset.$  By Theorem~\ref{Thurston} we may choose a representative
of each element of $G$ which permutes the elements of $\U$ and which is isotopic to
a generalized Dehn twist on each $A_i$.  Note that there may be some $A_i$ and
some $g$ such that the restriction to $A_i$ is inessential.

Let $U = \bigcup_i A_i$.  
By Theorem~\ref{Thurston}  elements of $G$ may be represented
by a finite set $\G_0$ of homeomorphisms of $M$ whose 
each of whose restrictions to $M \setminus int(U)$
is a finite order isometry in a hyperbolic metric. 

By Lemma~\ref{Kerckhoff} we may choose the representatives $\G_0$ so
their restrictions to $M \setminus int(U)$ form a finite group of
isometries of a hyperbolic metric.  We may then choose $\T^1$
coordinates on $\bigcup_i \partial A_i$ so that the restriction of any
element of $\G_0$ to $\partial A_i$ is a rotation.

For each $i$ we extend these coordinates on $\partial A_i$ to $\T^1
\times I$ coordinates on $A_i.$ For each $g_0 \in \G_0$ we observe
that $g_0|_{A_i}$ is isotopic rel $\partial A_i$ to a generalized Dehn
twist from $A_i$ to $g_0(A_i)$.  Applying Lemma~\ref{Dehn twist} we
alter each $g_0$ by an isotopy supported on $int(U)$ to obtain a new
homeomorphism $g: M \to M$ which is a generalized Dehn twist on each
$A_i.$ We denote by $\G$ the set of elements obtained in this way from
all the elements of $\G_0$.  It follows from the uniqueness and
composition properties in Lemma~\ref{Dehn twist} that elements of $\G$
restricted to $U$ form a group.  Since the restrictions of $\G$ and
$\G_0$ coincide on $M \setminus int(U)$ we conclude that $\G$ is a
subgroup of $\Homeo_0(M).$ Properties (1)-(4) are satisfied by
construction.
\end{proof}

\begin{cor}\label{finitely generated}
Let $M, X,$ and $G$ be as in Lemma~\ref{canonical} and suppose
$H$ is a normal subgroup of $G$.  Then $H$ is finitely generated.
\end{cor}

\begin{proof}
Let $\U, U$ and $\G$ be as in Lemma~\ref{canonical},
let $\H$ be the subgroup of $\G$ whose elements are representatives
of elements of $H$ and let $\K$ be the 
(finite index) normal subgroup of $\G$ which fixes each component of 
$M \setminus U$ pointwise.
Then $\H \cap \K$ fixes each $U_i \in \U$ and
restricts to Dehn twists on that $U_i.$
It follows that $\H \cap \K$ is
a finitely generated free abelian group.  Since 
$\H/\H \cap \K \cong \H\K/\K$ and $\H\K/\K$ is subgroup of
$\G/\K$ which is finite we conclude $\H$ is finitely generated.
\end{proof}

\section{The case of $M$ with genus $0$}

We are now prepared to present the proof of Theorem~\ref{solvable-0}.
\medskip

{\bf Theorem~\ref{solvable-0}} {\em 
  Suppose $M$ is an oriented compact connected surface of genus $0$
  and let $G$ be a finitely generated infinite subgroup of $\B(M, X)$
  with zero entropy, where $X$ is a non-empty finite $G$-invariant
  set. Then if $\partial M \ne \emptyset,$ or there is a point $x \in X$
  fixed by all elements of $G$, then $G$ is solvable. If $M = S^2$
  then there is a finite index normal subgroup $G_0$ of $G$ which is
  solvable and such that $G/G_0$ acts effectively on $S^2.$ Moreover
  when $G$ is solvable the kernel $K$ of the homomorphism
  $\pi: G \to \cS_{G}(X)$ is free abelian.
 }

\begin{proof}
We may assume that the Euler characteristic of $M \setminus X$ is negative
since otherwise $X$ contains at most two points and $\B(M, X)$ is
trivial or $\Z/2\Z$. Let $\U$ be the set of annuli guaranteed by Lemma~\ref{canonical}. The
set $\U$ must be non-empty since otherwise $G$ would be finite.


Recall that $U$ denotes the union of $A_i \in \U$ and let $\E$ 
denote the components of the complement of $U$.
The elements of $\E$ are permuted by $\G$.  We will prove the result
by induction on $\card(\E).$ If $\card(\E) = 1$ then 
each $A_i \in \U$ is a neighborhood of a boundary component of 
$M$ and hence invariant under $\G$.  It follows that $[\G,\G]$ acts
trivially on $U$.  By Lemma~\ref{canonical} $[\G,\G]$ acts as a
finite abelian group on $M \setminus int(U)$ and it must then be
trivial since it pointwise fixes the (non-empty) boundary. This provides
the base case of our induction.

We form a graph $\Gamma$ with one vertex for each element of $\E$ and
vertices joined by an edge if the corresponding elements of $\E$
intersect a common annulus $A_i \in \U$.  The graph $\Gamma$ is a tree
since a cycle in it would give rise to a curve in $M$ which
intersected an essential curve in some $A_i$ in a single point.

Proposition~\ref{tree action} implies there is an element $E_0 \in \E$
which is preserved by the action of $\G$. Note this is trivial if
there is $x \in X$ fixed by all elements of $\G$, because $E_0$ can be
taken as the element of $\E$ containing the global fixed point $x$.

Let $\G_0$ be the finite index normal subgroup of
$\G$ which preserves each component of $\partial E_0$ and
let $\G_0^{(n)} = [\G_0^{(n-1)},\G_0^{(n-1)}]$ for $n \ge 1.$

By Lemma~\ref{canonical} (4), we know that the action of $\G$
restricted to $E_0$ is finite.  This finite group (which we may assume
are isometries) is isomorphic to $\G/\G_0$ because the action of $G_0$
on $E_0$ is trivial since $G_0$ preserves components of $\partial E_0$
and any finite order element of $\Homeo_0(E_0)$ preserving $\partial
E_0$ must be the identity.  Note that this implies the group $\G/\G_0$
acts on $E_0$ and hence there is an effective action of the finite
group $G/G_0 \cong \G/\G_0$ on $S^2$. Also if there is a global fixed
point $x \in X \cap E_0$ this finite group $\G/\G_0$ is abelian.

Let $M_0$ be the compact surface $M \setminus int(E_0).$ Let $\bar A$ be
the union of the $A_i$'s which intersect the boundary of $E_0$.  Then
the restriction of $\G_0$ to the annuli $\bar A$ leaves each annulus
invariant and is abelian.  Hence the action of $\G_0^{(1)} = [\G_0, \G_0]$  
restricted to  $\bar \A$ is
trivial. Thus $\G_1$ preserves each of the components of $M_0$.  The
number of components of the complement of reducing curves for each of
these actions is smaller than the number of components 
for the action of $\G_0$.  Each
of these actions is an action on the disk $D^2$ whose restriction to
the boundary $\partial D^2$ has finite order.  If we cone off the
action on $\partial D^2$ we obtain an action of $\G_0^{(1)}$ with a global
fixed point $x_0$ the cone point. We now augment $X$ by adding $x_0$
and note that by the induction hypothesis the restriction of $\G_0^{(1)}$ to
a component of $M_0$ is solvable.  It follows that $\G_0^{(1)}$ is solvable
and consequently $\G_0$ is solvable. In the case that there is an $x
\in X$ fixed by all elements of $\G$ the group $\G/\G_0$ is abelian
and hence $[\G,\G] \subset \G_0$ and $\G$ is solvable.  Since $G \cong
\G$ we have completed the case that $M = S^2.$

We are left with the case that $\partial M \ne \emptyset$. In this
case we cone the action of $G$ on each boundary component of $M.$
Since $\G$ pointwise fixed $\partial M$ all the disks added in the
coning are fixed pointwise by $G$.  We may therefore augment $X$ by
adding two points in each of the added disks to $X$ to form $X'.$ This
guarantees that the natural map $B(M, X) \to B(S^2, X')$ is injective
on $G$ so applying the result for the sphere we obtain the result for
a subgroup $G$ of $B(M, X).$

To obtain the fact that $K$ is free abelian when $G$ is solvable we note
that elements of $K$ fix $X$ pointwise.  Each
element of $K$ permutes elements of $\U$, but each such element bounds
a disk containing a point of $X$. It follows that each element of $K$
preserves each element of $\U$ and is the identity on the complement
of their union $U$. Since the restriction on each element of $K$ to
an element of $\U$ is a Dehn twist, it follows that $K$ is free abelian.
\end{proof}

{\bf Corollary~\ref{cor-0}}
{\em Suppose $G$ satisfies the hypothesis of Theorem~\ref{solvable-0}.
If $G$ is solvable then $G/K \cong 
\cS_{G}(X).$  In particular the derived length satisfies 
\[
\dlen(G) -1 \le \dlen(\cS_{G}(X)) \le \dlen(G).
\]
}

\begin{proof}
We note that 
$\cS_G(X)$ is a homomorphic image of $G$ so $\dlen(\cS_G(X)) \le
\dlen(G).$  
If $n = \dlen(\cS_{G}(X))$ it is clear that $G^{(n)} \subset K$.
Hence since $K$ is abelian $G^{(n+1)}$ is trivial, so 
$\dlen(G) \le n +1$.  Thus
\[
\dlen(G) -1 \le \dlen(\cS_{G}(X)) \le \dlen(G).
\]
\end{proof}

\section{The case of  positive genus.}

Let $\U$ and $\G$ be as in the conclusion of Lemma~\ref{canonical}.
Recall $U = \bigcup_i A_i$ and let $\E$ be the components of $M
\setminus int(U).$ We define $\U_e$ to be the subset of $\U$
consisting of those $A_i \in \U$ which are essential in $M$ (not just in 
$M \setminus X$) and we let $\U_c$ be $\U \setminus \U_e$, i.e., those
elements of $\U$ which are inessential in $M$.  The core of an element
of $\U_c$ bounds a disk in $M$ containing at least two points of $X$.

\begin{lemma}\label{preserve U}
Suppose $M$ has genus $g > 1$, and  $\partial M = \emptyset$ 
Let  $G$ be a finitely generated
subgroup of $\B(M, X)$ each of whose elements have zero entropy. Then
\begin{enumerate}

\item If $\U_c \ne \emptyset$  and $\U_e = \emptyset$ then there is a normal 
subgroup of index $\le 84(g-1)$ in $\G$ 
which leaves invariant an element $A$ of $\U$.
\item If $\U_e \ne \emptyset$ then there is a subgroup of
$\G$ which leaves invariant an element $A$ of $\U$ and whose index
in $\G$ is $\le 84(g-1)(2g-2)!$.

\item If $\U_c = \U_e = \emptyset$ then $\G$ is finite an has 
order at most $84(g-1)$.

\end{enumerate}
\end{lemma}

\begin{proof}

If $\U_c \ne \emptyset$ and $\U_e = \emptyset$ 
then we let $\U_m$ be the (non-empty) subset of maximal
elements of $\U_c$ under inclusion in disks.  
More precisely, $A \in \U_m$ if $A \in \U$
and one component of the complement of  $A$ is a 
disk in $M$ and $A$ is not a subset of an open
disk which is a component of the complement of some other $A' \in \U$.
There is only one component $C$ of the complement of the union of all elements
of $\U_m$ which has positive genus (necessarily of genus $g$) since there are no
elements of $\U_m$ which are essential in $M$.  The group $\G$ must preserve
the surface $C$ and its action on $C$ must be irreducible since $\U_e = \emptyset$.
Hence the action of $\G$ on $C$ is finite. Since $C$ has genus $g$
the order of the restriction of $\G$ to $C$ is
at most $84(g-1)$ (see Chapter 7 of \cite{FM}). It follows that the
normal subgroup of $\G$ which leaves invariant each of the elements of $\U_m$ which
intersects a boundary component of $C$ has index at most $84(g-1)$. This completes (1).

If $\U_e \ne \emptyset$ then
$\G$ acts in a way that it permutes the elements of $\U_e$ and 
permutes the components of the complement of their union.  
A component of the complement of $\U_e$ may be an annulus containing points
of $X$, but there must be at least one non-annular component.  There are
at most $2g-2 = |\chi(M)|$ non-annular components (see \cite{H}) and they
are permuted by $\G$.
Hence there is a subgroup $\G'$ of $\G$ which preserves each of these components
and which has index at most $(2g-2)!$ in $\G$. Any such non-annular component $E$
has genus at most $g$. The action of $\G'$ on $E$ can be extended to
$\hat E$ which we define to be $E$ with any boundary components coned off.
Since there are no elements of $\U_e$ in $E$, we may apply 
part (1) of this lemma to the group $\G'$ acting on $\hat E$ we conclude
there is a normal subgroup $\G''$ of $\G'$ of index 
at most $84(g-1)$ fixing at least one element of $\U$.  This completes
(2).

Finally if $\U_e = \U_c = \emptyset$ then $\U = \emptyset$ so $\G$ is finite 
by Lemma~\ref{canonical} and has order at most $84(g-1)$
(see Chapter 7 of \cite{FM}).

\end{proof}

\begin{lemma}\label{induction}
Suppose $M$ is an oriented compact connected surface of genus $g \ge
1$ and let $G$ be a finitely generated subgroup of
$\B(M, X)$ which contains no element with positive entropy and is
realized by $\G \cong G$ as in Lemma~\ref{canonical}.  If any 
of the following hold:
\begin{enumerate}
\item $\partial M \ne \emptyset$, or
\item $\Fix(\G) \cap X \ne \emptyset$, or
\item there exists an $A \in \U$ which is $\G$-invariant,
\end{enumerate}
then the group $G$ is solvable. 
\end{lemma}

\begin{proof}
We induct on the genus $g$ of $M$ and the cardinality of $\U$.
More precisely given $M$ and $X$ we will show the result holds if it holds
for surfaces of lower genus and for surfaces of the same genus with a 
lower cardinality of $\U$.  We first address the base cases 
of the induction. If the genus $g=0$ this is provided by Theorem~\ref{solvable-0}.
If $\U = \emptyset$ Lemma~\ref{canonical} implies that $\G$ is finite and
we consider the three possible hypotheses.  If $\partial M \ne \emptyset$
and $\G$ fixes it pointwise then $\G$ is trivial.  
If $x \in \Fix(\G) \cap X$ then $\G$ acts in an abelian fashion on a neighborhood
of $x$ and hence on all of $M$.  If there exists an $A \in \U$ which is $\G$-invariant,
then the restriction of $G^{(1)} = [G,G]$ to $A$ is trivial and hence
$G^{(1)}$ acts trivially on $M$. Hence in all cases $G$ is solvable. Thus the
result holds if the genus $g$ is zero or the cardinality of $\U$ is zero.

We proceed to the inductive step, first in the 
case that $\partial M \ne \emptyset$. Let $B$ be a 
component of $\partial M$. The boundary component $B$ may be a boundary component
of an annulus $A \in \U$ (which collars $B$). Let $E_1$
be the component of $\E = M \setminus U$ which intersects $B$ or an annulus
$A \in \U$ which contains $B$.  The group $\G^{(1)} = [\G, \G]$ acts as the
identity on $B$ and hence on one component of the boundary of $E_1$.  Since
every element of $\G^{(1)}$ acts as an irreducible map on $E_1$ and fixes one component
of its boundary pointwise, $\G^{(1)}$ must act trivially on $E_1.$ If we consider
the components of the complement of $E_1$ the group $\G^{(1)}$ preserves them and each
of them has either lower genus or contains only a proper subset of $\U_c$ or both.
The induction hypothesis then implies that $\G^{(1)}$ restricted to these components
is solvable. This completes the proof part (1), namely that $\G$ is solvable when 
$\partial M \ne \emptyset$.  

If there exists $x \in X$ fixed by each element
of $\G$ then we may blow it up to obtain a surface $M'$ with boundary on which
$\G$ acts.  Since $x$ lies in the interior of the complement of $U$ the action
of $\G$ on $\partial M'$ is abelian.  The previous case applied to $M'$
then shows that $\G^{(1)}$ is solvable so $\G$ is also.

We now consider the case that $\partial M = \emptyset$.
Let $\U$ and $\G$ be as in the conclusion of Lemma~\ref{canonical} and
let $\G_e$ and $\U_e$  be as in Lemma~\ref{preserve U}. 
By hypothesis there is an $A \in \U$ which
is $\G$ invariant. The group $\G^{(1)}$ acts trivially on $A$.
Let $M'' = M \setminus int(A)$ so each component of $M''$ is $\G^{(1)}$ invariant
and $\G^{(1)}$ fixes the components of $\partial M''$ pointwise. $\G^{(1)}$ acts
on each component of $\M''$ fixing the boundary pointwise and hence
this action is solvable by part (1) of this theorem.
\end{proof}

We can now complete the proofs of Theorem~\ref{solvable-1} 
and Theorem~\ref{solvable-T}.  Recall that in both of these results
$X$ is a non-empty finite set.
\medskip

\noindent
{\bf Theorem~\ref{solvable-1}}
{\em Suppose $M$ is an oriented compact connected surface of genus $g >
1$ and let $G$ be a finitely generated infinite subgroup of $\B(M, X)$ which
contains no element with positive entropy. 
\begin{enumerate}
\item If $\partial M \ne \emptyset$, or there exists $x \in X$ fixed by each
element of $G$, then $G$ is solvable.
\item Otherwise $G$ contains a solvable 
subgroup $G_0$ of index at most $84(g-1)(2g-2)!$ in $G$.
\end{enumerate}
}

\begin{proof}
If $\partial M \ne \emptyset$ then the result follows from Lemma~\ref{induction}.
If $\partial M = \emptyset$ and $\U_c \ne \emptyset$ then by part (1) of 
Lemma~\ref{preserve U} there is a subgroup of index 
$\le 84(g-1)$ in $\G$ which leaves invariant an element $A$ of $\U$. This
subgroup is solvable by Lemma~\ref{induction}. 

If $\partial M = \emptyset$ and $\U_c = \emptyset$ and 
$\U_e \ne \emptyset$ then by part (2) of 
Lemma~\ref{preserve U} there is a subgroup of index  
at most $84(g-1)(2g-2)!$ in $\G$ which leaves invariant an element $A$ of $\U$. This
subgroup is solvable by Lemma~\ref{induction}. 

Finally, if $\U_e = \U_c = \emptyset$ then part (3) of 
Lemma~\ref{preserve U} implies $\G$ is finite 
contradicting our hypothesis.
\end{proof}

Only the proof of Theorem \ref{solvable-T} remains.
\medskip

\noindent
{\bf Theorem \ref{solvable-T}}
{\em Suppose $M$ is an oriented compact connected surface of genus $g =1$ 
and let $G$ be a finitely generated subgroup of $\B(M, X)$ which
contains no element with positive entropy. Then $G$ is solvable.
}

\begin{proof}
If $\partial M \ne \emptyset$ the result follows from Lemma~\ref{induction},
so we may assume $\partial M = \emptyset$ and $M = \T^2$.
As before let $\U_e$ denote the subset of elements of $\U$ which are essential 
in $M$ and let $\G_e$ be the subgroup of $\G$ which preserves each element of $\U_e$.
Then all elements of $\U_e$ must be parallel in $M$ and
separated by points of $X$. We may choose
a simple closed curve $\gamma$ in $M \setminus X$ crossing each $U_i \in \U_e$ exactly
once.  A circular order on $\gamma$ induces a circular order on the
components of the complement of annuli in $\U_e$ 
or equivalently on the elements of $\U_e$.
This order must  must be preserved by elements of $\G$.  It follows that
the action of $\G$ on $\U_e$ factors through a finite cyclic group and that
$\G/\G_e$ is a finite cyclic group.   We conclude that
$\G^{(1)} = [\G,\G]$ preserves all elements of $\U_e$. We may therefore
apply Lemma~\ref{induction} to $\G^{(1)}$ to obtain the desired result.
\end{proof}

\subsection*{Acknowledgements}
The authors would like to thank Chris Leininger for reading an earlier draft of this paper and providing several useful comments and suggestions.

\end{document}